\documentclass[a4paper, 10pt]{article}

\usepackage[latin1]{inputenc}
\usepackage{subfigure}
\usepackage{amsmath}
\usepackage{amssymb}
\usepackage{amsthm}
\usepackage{amsfonts}
\usepackage[cyr]{aeguill}
\usepackage{xspace}
\usepackage[english]{babel}
\usepackage{graphicx}
\usepackage{float}
\usepackage{mathtools}
\usepackage{psfrag,floatflt,hyperref}
\usepackage{verbatim}
\usepackage{geometry}
\usepackage{multirow}
\usepackage[retainorgcmds]{IEEEtrantools}
\usepackage{color}
\newtheorem{theorem}{Theorem}[section]

\newtheorem{corollary}[theorem]{Corollary}

\newenvironment{remark}[1][Remark.]{\begin{trivlist}
\item[\hskip \labelsep {\bfseries #1}]}{\end{trivlist}}
\newenvironment{example}[1][Examples.]{\begin{trivlist}
\item[\hskip \labelsep {\bfseries #1}]}{\end{trivlist}}

\title{Spanning trees in directed circulant graphs and cycle power graphs\footnote{The author acknowledges support from the Swiss NSF grant $200021\_132528/1$.}}
\author{Justine Louis}
\date{$10$th July $2015$}

\DeclareMathOperator{\argcosh}{Argcosh}
\DeclareMathOperator{\Arcsin}{Arcsin}
\DeclareMathOperator{\Arctg}{Arctg}

\newcommand*{\eqdef}{=\mathrel{\vcenter{\baselineskip0.5ex \lineskiplimit0pt
                     \hbox{\scriptsize.}\hbox{\scriptsize.}}}}
                     
\begin{document}
        \maketitle
\begin{abstract}
The number of spanning trees in a class of directed circulant graphs with generators depending linearly on the number of vertices $\beta n$, and in the $n$-th and $(n-1)$-th power graphs of the $\beta n$-cycle are evaluated as a product of $\lceil\beta/2\rceil-1$ terms.
\end{abstract}

\section{Introduction}
In this paper we study the number of spanning trees in a class of directed and undirected circulant graphs. Let $1\leqslant\gamma_1\leqslant\cdots\leqslant\gamma_d\leqslant\lfloor n/2\rfloor$ be positive integers. A circulant directed graph, or circulant digraph, on $n$ vertices generated by $\gamma_1,\ldots,\gamma_d$ is the directed graph on $n$ vertices labelled $0,1,\ldots,n-1$ such that for each vertex $v\in\mathbb{Z}/n\mathbb{Z}$ there is an oriented edge connecting $v$ to $v+\gamma_m$ mod $n$ for all $m\in\{1,\ldots,d\}$. We will denote such graphs by $\overrightarrow{C}^{\gamma_1,\ldots,\gamma_d}_n$. Similarly, a circulant graph on $n$ vertices generated by $\gamma_1,\ldots \gamma_d$, denoted by $C^{\gamma_1,\ldots,\gamma_d}_n$, is the undirected graph on $n$ vertices labelled $0,1,\ldots,n-1$ such that each vertex $v\in\mathbb{Z}/n\mathbb{Z}$ is connected to $v\pm\gamma_m$ mod $n$ for all $m\in\{1,\ldots,d\}$. Circulant graphs and digraphs are used as models in network theory. In this context, they 
are called multi-loop networks, or double-loop networks when they are $2$-generated, see for example \cite{MR1846929,MR1973148}. The number of spanning tree measures the reliability of a network.\\
The evaluation of the number of spanning trees in circulant graphs and digraphs has been widely studied, were both exact and asymptotic results have been obtained as the number of vertices grows, see \cite{MR2565193,MR2574828,louis2015asymptotics,louis2015formula,MR2445039} and references therein. In \cite{MR2261780,MR2320194}, the authors showed that the number of spanning trees in such graphs satisfy linear recurrence relations. Yong, Zhang and Golin developped a technique in \cite{MR2445039} to evaluate the number of spanning trees in a particular class of double-loop networks $\overrightarrow{C}^{p,\gamma n+p}_{\beta n}$. In the first section of this work, we derive a closed formula for these graphs, and more generally for $d$-generated circulant digraphs with generators depending linearly on the number of vertices, that is $\overrightarrow{C}^{p,\gamma_1n+p \ldots,\gamma_{d-1}n+p}_{\beta n}$ where $p,\gamma_1,\ldots,\gamma_{d-1},\beta,n$ are positive integers. This partially answers an open question 
posed in \cite{MR2565193} by simplifying the formula given in \cite[Corollary $1$]{MR2565193}.\\ 
In the second section we calculate the number of spanning trees in the $n$-th and $(n-1)$-th power graphs of the $\beta n$-cycle which are the circulant graphs generated by the $n$, respectively $n-1$, first consecutive integers, denoted by $\boldsymbol{C}^n_{\beta n}$ and $\boldsymbol{C}^{n-1}_{\beta n}$ respectively, where $\beta\in\mathbb{N}_{\geqslant2}$. As a consequence, the asymptotic behaviour of it is derived. Cycle power graphs appear, for example, in graph colouring problems, see \cite{MR2587027,MR1720404}.\\
The results obtained here are derived from the matrix tree theorem (see \cite{MR2339282,MR1271140}) which provides a closed formula of a product of $\beta n-1$ terms for a graph on $\beta n$ vertices. Our formulas are a product of $\lceil\beta/2\rceil-1$ terms and are therefore interesting when $n$ is large. In both cases, the symmetry of the graphs is reflected in the formulas which are expressed in terms of eigenvalues of subgraphs of the original graph. This fact was already observed in \cite{louis2015formula}.
\par\vspace{\baselineskip}
\noindent
\textbf{Acknowledgements:} The author thanks Anders Karlsson for reading the manuscript and useful discussions.
\section{Spanning trees in directed circulant graphs}
Let $G$ be a directed graph and $V(G)$ its vertex set. A spanning arborescence converging to $v\in V(G)$ is an oriented subgraph of $G$ such that the out-degree of all vertices except $v$ equals one, and the out-degree of $v$ is zero. We define the combinatorial Laplacian of a directed graph $G$ as an operator acting on the space of functions defined on $V(G)$, by
\begin{equation}
\label{Delta-}
\Delta^-_Gf(x)=\sum_ {y:\ x\rightarrow y}(f(x)-f(y))
\end{equation}
where the sum is over all vertices $y$ such that there is an oriented edge from $x$ to $y$. Equivalently, the combinatorial Laplacian can be defined as a matrix by $\Delta^-_G=D^--A$, where $D^-$ is the out-degree matrix and $A$ is the adjacency matrix such that $(A)_{ij}$ is the number of directed edges from $i$ to $j$. Let $\tau^-(G,v)$ denote the number of arborescences converging to $v$. The Tutte matrix tree theorem (see \cite{MR2339282}) states that for all $v\in V(G)$,
\begin{equation*}
\tau^-(G,v)=\det\Delta^-_{G,v}
\end{equation*}
where $\det\Delta^-_{G,v}$ is the $v$-th cofactor of the Laplacian $\Delta^-_G$ obtained by deleting the row and column of $\Delta^-_G$ corresponding to the vertex $v$. For a regular directed graph $G$, we define the number of spanning trees in $G$, $\tau(G)$, by the sum over all vertices $v\in V(G)$ of the number of arborescences converging to $v$, that is
\begin{equation*}
\tau(G)=\sum_{v\in V(G)}\tau^-(G,v).
\end{equation*}
Notice that we could have defined the number of spanning trees by the sum over all vertices $v\in V(G)$ of the number of spanning arborescences diverging from $v$.\\
By symmetry, all cofactors of the Laplacian of a directed circulant graph are equal and are equal to the product of the non-zero eigenvalues of the Laplacian divided by the number of vertices. Therefore we have that
\begin{equation*}
\tau(G)=\prod_{k=1}^{\lvert V(G)\rvert}\lambda_k
\end{equation*}
where $\lambda_k$, $k=1,\ldots,\lvert V(G)\rvert$, denote the non-zero eigenvalues of the Laplacian of $G$. The non-zero eigenvalues of the Laplacian of the directed circulant graph $\overrightarrow{C}^{\gamma_1,\ldots,\gamma_d}_n$ are given by (see \cite[Proposition $3.5$]{MR1271140})
\begin{equation*}
\lambda_k=d-\sum_{m=1}^de^{2\pi i\gamma_mk/n},\quad k=1,\ldots,n-1.
\end{equation*}
This can also be derived by noticing that the eigenvectors are given by the characters $\chi_k(x)=e^{2\pi ikx/n}$, $k=0,1,\ldots,n-1$, and then applying the Laplacian (\ref{Delta-}) on it.\\
In this section, we establish a formula for the number of spanning trees in directed circulant graphs $\overrightarrow{C}^\Gamma_{\beta n}$ generated by $\Gamma=\{p,\gamma_1n+p,\ldots,\gamma_{d-1}n+p\}$ and in the particular case of two generators $\overrightarrow{C}^{p,\gamma n+p}_{\beta n}$. Figure \ref{directedgraphs} illustrates a $2$ and a $3$ generated directed circulant graphs. We denote by $\mu_k=d-1-\sum_{m=1}^{d-1}e^{2\pi i\gamma_mk/\beta}$, $k=1,\ldots,\beta-1$, the non-zero eigenvalues of the Laplacian on the directed circulant graph $\overrightarrow{C}^{\gamma_1,\ldots,\gamma_{d-1}}_\beta$ and by $\eta_k=2(d-1)-2\sum_{m=1}^{d-1}\cos(2\pi\gamma_mk/\beta)$, $k=1,\ldots,\beta-1$, the non-zero eigenvalues of the Laplacian on the circulant graph $C^{\gamma_1,\ldots,\gamma_{d-1}}_\beta$. Let $A$ be a statement and $\delta_A$ be defined by
\begin{equation*}
\delta_A=\left\{\begin{array}{rl}1&\textnormal{if }A\textnormal{ is satisfied}\\0&\textnormal{otherwise}\end{array}.\right.
\end{equation*}
\begin{figure}[!ht]
\centering
\subfigure[$\protect\overrightarrow{C}^{2,n}_{3n}$ with $n=5$]{\includegraphics[width=5cm]{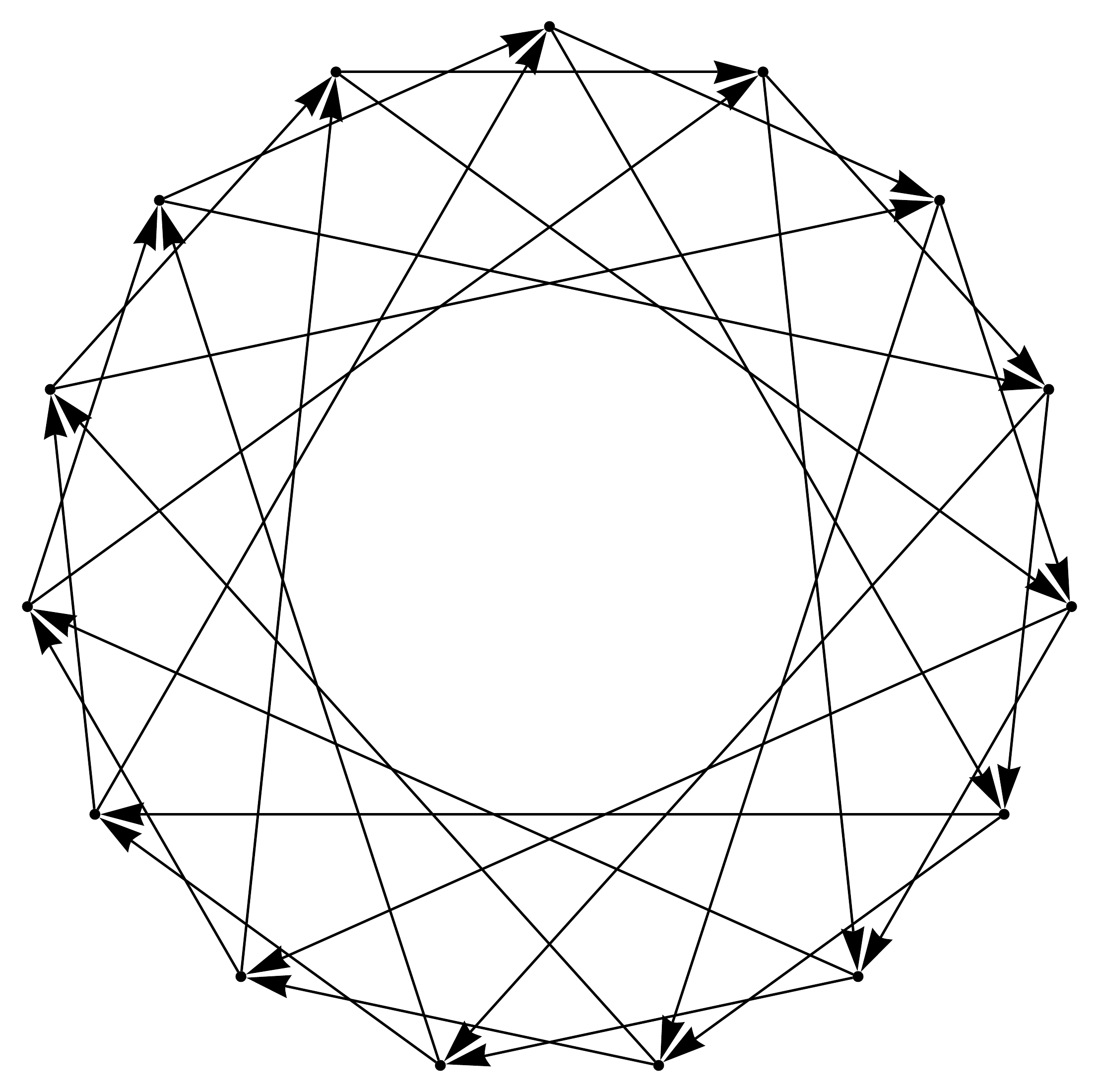}}
\hspace{2cm}
\subfigure[$\protect\overrightarrow{C}^{1,n+1,2n+1}_{4n}$ with $n=5$]{\includegraphics[width=5cm]{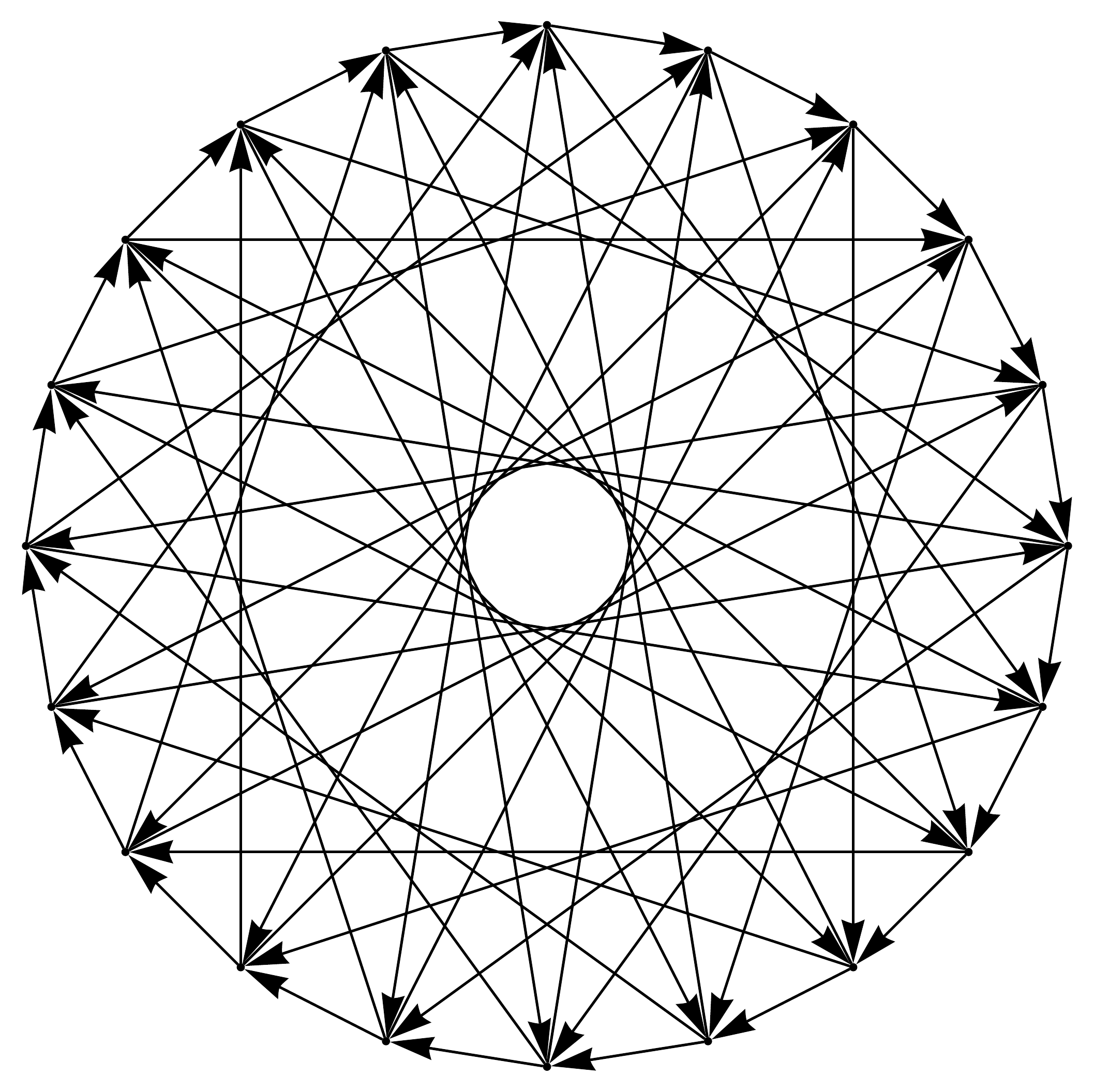}}
\caption{Examples of directed graphs}
\label{directedgraphs}
\end{figure}
\begin{theorem}
\label{dicd}
Let $1\leqslant\gamma_1\leqslant\cdots\leqslant\gamma_{d-1}\leqslant\beta$ and $p$, $n$ be positive integers. For all even $n\in\mathbb{N}_{\geqslant2}$ such that $(p,n)=1$, the number of spanning trees in the directed circulant graph $\overrightarrow{C}^{\Gamma}_{\beta n}$, where $\Gamma=\{p,\gamma_1n+p,\ldots,\gamma_{d-1}n+p\}$, is given by
\begin{align*}
\tau(\overrightarrow{C}^{\Gamma}_{\beta n})&=nd^{\beta n-1}\Big(1-\delta_{\beta\textnormal{ even}}\frac{(-1)^p}{d^n}(1+\sum_{m=1}^{d-1}(-1)^{\gamma_m})^n\Big)\\
&\times\prod_{k=1}^{\lceil\beta/2\rceil-1}\left(1-2\Big|1-\frac{\mu_k}{d}\Big|^n\cos\left(\frac{2\pi pk}{\beta}+n\Arctg\left(\frac{\sum_{m=1}^{d-1}\sin(2\pi\gamma_mk/\beta)}{d-\eta_k/2}\right)\right)+\Big|1-\frac{\mu_k}{d}\Big|^{2n}\right)
\end{align*}
and for odd $n\in\mathbb{N}_{\geqslant1}$,
\begin{align*}
&\tau(\overrightarrow{C}^{\Gamma}_{\beta n})=nd^{\beta n-1}\Big(1-\delta_{\beta\textnormal{ even}}\frac{(-1)^p}{d^n}(1+\sum_{m=1}^{d-1}(-1)^{\gamma_m})^n\Big)\\
&\times\prod_{k=1}^{\lceil\beta/2\rceil-1}\left(1-2\textnormal{sgn}(d-\eta_k/2)\Big|1-\frac{\mu_k}{d}\Big|^n\cos\left(\frac{2\pi pk}{\beta}+n\Arctg\left(\frac{\sum_{m=1}^{d-1}\sin(2\pi\gamma_mk/\beta)}{d-\eta_k/2}\right)\right)+\Big|1-\frac{\mu_k}{d}\Big|^{2n}\right)
\end{align*}
where $\lceil x\rceil$ is the smallest integer greater or equal to $x$, $\lvert.\rvert$ denotes the modulus and we set $\textnormal{sgn}(0)=1$. The number of spanning trees in $\overrightarrow{C}^{\Gamma}_{\beta n}$ is zero if either $(p,n)=1$ and $\beta$, $p$, $\gamma_m$, $m=1,\ldots,d-1$ are all even or either $(p,n)\neq1$.
\end{theorem}
\begin{proof}
From the Tutte matrix tree theorem, the number of spanning trees in $\overrightarrow{C}^\Gamma_{\beta n}$ is given by
\begin{equation*}
\tau(\overrightarrow{C}^{\Gamma}_{\beta n})=\prod_{k=1}^{\beta n-1}(d-e^{2\pi ipk/(\beta n)}-\sum_{m=1}^{d-1}e^{2\pi i(\gamma_mn+p)k/(\beta n)}).
\end{equation*}
By splitting the product over $k=1,\ldots,\beta n-1$ into two products, when $k$ is a multiple of $\beta$, that is $k=l\beta$ with $l=1,\ldots,n-1$, and over non-multiples of $\beta$, that is, $k=k'+l'\beta$ with $k'=1,\ldots,\beta-1$ and $l'=0,1,\ldots,n-1$, we have
\begin{equation}
\label{tau}
\tau(\overrightarrow{C}^{\Gamma}_{\beta n})=\prod_{l=1}^{n-1}(d-de^{2\pi ipl/n})\prod_{k=1}^{\beta-1}\prod_{l'=0}^{n-1}(d-(1+\sum_{m=1}^{d-1}e^{2\pi i\gamma_mk/\beta})e^{2\pi ipk/(\beta n)}e^{2\pi ipl'/n}).
\end{equation}
We have that
\begin{equation*}
\prod_{l=1}^{n-1}(d-de^{2\pi ipl/n})=d^{n-1}\prod_{l=1}^{n-1}(1-e^{2\pi ipl/n})=nd^{n-1}\delta_{(p,n)=1}.
\end{equation*}
This equality comes from the fact that $\prod_{l=1}^{n-1}(1-e^{2\pi ipl/n})$ is the number of spanning trees of the directed graph $\overrightarrow{C}^p_n$, which is isomorphic to the directed cycle on $n$ vertices if $(p,n)=1$, and is not connected if $(p,n)\neq1$. Therefore the product is equal to $n\delta_{(p,n)=1}$.\\
Hence, if $(p,n)\neq1$, we have
\begin{equation*}
\tau(\overrightarrow{C}^{\Gamma}_{\beta n})=0.
\end{equation*}
Let $p$ be relatively prime to $n$. Using that the complex numbers $e^{2\pi il/n}$, $l=0,1,\ldots,n-1$, are the $n$ non-trivial roots of unity, we have for all $x$,
\begin{equation}
\label{unityroots}
\prod_{l=0}^{n-1}(x-e^{2\pi ilp/n})=x^n-1.
\end{equation}
since $(p,n)=1$. Equivalently we have,
\begin{equation*}
\prod_{l=0}^{n-1}(1-xe^{2\pi ilp/n})=1-x^n.
\end{equation*}
Using this identity in (\ref{tau}) enables to evaluate the product over $l'$, it comes
\begin{equation}
\label{betaproduct}
\tau(\overrightarrow{C}^{\Gamma}_{\beta n})=nd^{\beta n-1}\prod_{k=1}^{\beta-1}(1-\frac{1}{d^n}(1+\sum_{m=1}^{d-1}e^{2\pi i\gamma_mk/\beta})^ne^{2\pi ipk/\beta}).
\end{equation}
For odd $\beta$ we write the product over $k$, $k=1,\ldots,\beta-1$, as a product from $1$ to $(\beta-1)/2$, and for even $\beta$ we write it as a product from $1$ to $\beta/2-1$ and add the $k=\beta/2$ factor which is given by $1-(-1)^p(1+\sum_{m=1}^{d-1}(-1)^{\gamma_m})^n/d^n$. Writing the above expression in terms of \linebreak$\mu_k=d-1-\sum_{m=1}^{d-1}e^{2\pi i\gamma_mk/\beta}$, it comes
\begin{align}
\tau(\overrightarrow{C}^{\Gamma}_{\beta n})&=nd^{\beta n-1}\Big(1-\delta_{\beta\textnormal{ even}}\frac{(-1)^p}{d^n}(1+\sum_{m=1}^{d-1}(-1)^{\gamma_m})^n\Big)\nonumber\\
&\times\prod_{k=1}^{\lceil\beta/2\rceil-1}(1-(1-\mu_k/d)^ne^{2\pi ipk/\beta})(1-(1-\mu_k^\ast/d)^ne^{-2\pi ipk/\beta})\nonumber\\
&=nd^{\beta n-1}\Big(1-\delta_{\beta\textnormal{ even}}\frac{(-1)^p}{d^n}(1+\sum_{m=1}^{d-1}(-1)^{\gamma_m})^n\Big)\nonumber\\
&\times\prod_{k=1}^{\lceil\beta/2\rceil-1}(1-2\lvert1-\mu_k/d\rvert^n\cos(2\pi pk/\beta+n\phi_k)+\lvert1-\mu_k/d\rvert^{2n})
\label{tau2}
\end{align}
where $\phi_k$ is the phase of the complex number $1-\mu_k/d$ such that $1-\mu_k/d=\lvert1-\mu_k/d\rvert e^{i\phi_k}$. We have
\begin{equation*}
\lvert1-\mu_k/d\rvert=\frac{1}{d}\Big((d-\eta_k/2)^2+\Big(\sum_{m=1}^{d-1}\sin(2\pi\gamma_mk/\beta)\Big)^2\Big)^{1/2}
\end{equation*}
and
\begin{equation*}
\cos{\phi_k}=\frac{d-\eta_k/2}{\lvert d-\mu_k\rvert},\quad\sin{\phi_k}=\frac{\sum_{m=1}^{d-1}\sin(2\pi\gamma_mk/\beta)}{\lvert d-\mu_k\rvert}.
\end{equation*}
Therefore for $k$ such that $d-\eta_k/2\neq0$, the phase is given by
\begin{equation}
\label{phik}
\phi_k=\Arctg\left(\frac{\sum_{m=1}^{d-1}\sin(2\pi\gamma_mk/\beta)}{d-\eta_k/2}\right)+\epsilon\pi
\end{equation}
where $\epsilon=0$ if $\textnormal{sgn}(d-\eta_k/2)=1$ and $\epsilon\in\{-1,1\}$ if $\textnormal{sgn}(d-\eta_k/2)=-1$. For $k$ such that $d-\eta_k/2=0$, we take the limit as $d-\eta_k/2\rightarrow0$ in (\ref{phik}), with $\epsilon=0$. The theorem follows by putting equation (\ref{phik}) into equation (\ref{tau2}).\\
When $\beta$, $p$ and $\gamma_m$, $m=1,\ldots,d-1$ are all even, the directed circulant graph $\overrightarrow{C}^{\Gamma}_{\beta n}$ is not connected and therefore the number of spanning trees is zero, this is reflected in the formula.
\end{proof}
In the following theorem we state the particular case on two-generated directed circulant graphs.
\begin{theorem}
\label{d=2}
Let $1\leqslant\gamma\leqslant\beta$ and $p$, $n$ be positive integers. For odd $\beta$ and all $n\in\mathbb{N}_{\geqslant1}$ such that $(p,n)=1$, the number of spanning trees in the directed circulant graph $\overrightarrow{C}^{p,\gamma n+p}_{\beta n}$ is given by
\begin{equation*}
\tau(\overrightarrow{C}^{p,\gamma n+p}_{\beta n})=n2^{\beta n-1}\prod_{k=1}^{(\beta-1)/2}\Big(1-2\cos(2\pi(p+\gamma n/2)k/\beta)\cos^n(\pi\gamma k/\beta)+\cos^{2n}(\pi\gamma k/\beta)\Big)
\end{equation*}
and for even $\beta$, if $\gamma$ or $p$ is odd, then
\begin{equation*}
\tau(\overrightarrow{C}^{p,\gamma n+p}_{\beta n})=n2^{\beta n-1+\delta_{\gamma\textnormal{ even}}}\prod_{k=1}^{\beta/2-1}\Big(1-2\cos(2\pi(p+\gamma n/2)k/\beta)\cos^n(\pi\gamma k/\beta)+\cos^{2n}(\pi\gamma k/\beta)\Big).
\end{equation*}
The number of spanning trees in $\overrightarrow{C}^{p,\gamma n+p}_{\beta n}$ is zero if either $(p,n)=1$ and $\beta$, $p$ and $\gamma$ are all even or either $(p,n)\neq1$.
\end{theorem}
\begin{proof}
From equation (\ref{betaproduct}) it follows
\begin{equation*}
\tau(\overrightarrow{C}^{p,\gamma n+p}_{\beta n})=n2^{\beta n-1}\prod_{k=1}^{\beta-1}(1-e^{2\pi i(p+\gamma n/2)k/\beta}\cos^n(\pi\gamma k/\beta)).
\end{equation*}
For odd $\beta$, we have
\begin{align*}
\tau(\overrightarrow{C}^{p,\gamma n+p}_{\beta n})&=n2^{\beta n-1}\prod_{k=1}^{(\beta-1)/2}(1-e^{2\pi i(p+\gamma n/2)k/\beta}\cos^n(\pi\gamma k/\beta))\\
&\qquad\qquad\quad\qquad\times(1-e^{-2\pi i(p+\gamma n/2)k/\beta}\cos^n(\pi\gamma k/\beta))\\
&=n2^{\beta n-1}\prod_{k=1}^{(\beta-1)/2}(1-2\cos(2\pi(p+\gamma n/2)k/\beta)\cos^n(\pi\gamma k/\beta)+\cos^{2n}(\pi\gamma k/\beta)).
\end{align*}
For even $\beta$, the factor $k=\beta/2$ is added:
\begin{equation*}
1-e^{\pi i(p+\gamma n/2)}\cos^n(\pi\gamma/2)=\left\{\begin{array}{rl}0&\textnormal{if }p\textnormal{ and }\gamma\textnormal{ are even}\\1&\textnormal{if }\gamma\textnormal{ is odd}\\2&\textnormal{otherwise}\end{array}.\right.
\end{equation*}
For even $\beta$, $p$ and $\gamma$, the graph $\overrightarrow{C}^{p,\gamma n+p}_{\beta n}$ is not connected and therefore the number of spanning trees is zero. Therefore if $p$ or $\gamma$ is odd, we have
\begin{align*}
\tau(\overrightarrow{C}^{p,\gamma n+p}_{\beta n})&=n2^{\beta n-1+\delta_{\gamma\textnormal{ even}}}\prod_{k=1}^{\beta/2-1}(1-e^{2\pi i(p+\gamma n/2)k/\beta}\cos^n(\pi\gamma k/\beta))\\
&\quad\qquad\qquad\qquad\qquad\times(1-e^{-2\pi i(p+\gamma n/2)k/\beta}\cos^n(\pi\gamma k/\beta))\\
&=n2^{\beta n-1+\delta_{\gamma\textnormal{ even}}}\prod_{k=1}^{\beta/2-1}(1-2\cos(2\pi(p+\gamma n/2)k/\beta)\cos^n(\pi\gamma k/\beta)+\cos^{2n}(\pi\gamma k/\beta)).
\end{align*}
\end{proof}
\begin{example}
Consider the case when $p=\beta=3$ and $\gamma=2$. It follows from Theorem \ref{d=2} that $\tau(\overrightarrow{C}^{3,2n+3}_{3n})=0$ if $n$ is a multiple of $3$, otherwise,
\begin{align*}
\tau(\overrightarrow{C}^{3,2n+3}_{3n})&=n2^{3n-1}(1-2\cos(2\pi n/3)\cos^n(2\pi/3)+\cos^{2n}(2\pi/3))\\
&=n(2^{3n-1}-2^{2n}\cos(\pi n/3)+2^{n-1})
\end{align*}
as stated in \cite[Example $4$.(iii)]{MR2445039}. As another example, consider the case when $p=2$, $\gamma=5$ and $\beta=6$. From Theorem \ref{d=2}, for even $n$, $\tau(\overrightarrow{C}^{2,5n+2}_{6n})=0$, and for odd $n$,
\begin{align*}
\tau(\overrightarrow{C}^{2,5n+2}_{6n})&=n2^{6n-1}(1-2\cos(2\pi(2+5n/2)/6)\cos^n(5\pi/6)+\cos^{2n}(5\pi/6))\\
&\quad\times(1-2\cos(4\pi(2+5n/2)/6)\cos^n(10\pi/6)+\cos^{2n}(10\pi/6))\\
&=\frac{n}{2}(2^{3n}+2^{2n}3^{n/2}\cos(\pi n/6)-2^{2n}3^{(n+1)/2}\sin(\pi n/6)+6^n)\\
&\quad\times(2^{3n}-2^{2n-1}3^{n/2}\cos(\pi n/3)+2^{n-1}3^{(n+1)/2}\sin(\pi n/3)+2^n).
\end{align*}
\end{example}
\section{Spanning trees in cycle power graphs}
The $k$-th power graph of the $n$-cycle, denoted by $\boldsymbol{C}^k_n$, is the graph with the same vertex set as the $n$-cycle where two vertices are connected if their distance on the $n$-cycle is at most $k$. It is therefore the circulant graph on $n$ vertices generated by the first $k$ consecutive integers. In this section, we derive a formula for the number of spanning trees in the $n$-th and $(n-1)$-th power graph of the $\beta n$-cycle, where $\beta\in\mathbb{N}_{\geqslant2}$. As a consequence we derive the asymptotic behaviour of it as $n$ goes to infinity.\\
The combinatorial Laplacian of an undirected graph $G$ with vertex set $V(G)$ defined as an operator acting on the space of functions is
\begin{equation*}
\Delta_Gf(x)=\sum_{y\sim x}(f(x)-f(y))
\end{equation*}
where the sum is over all vertices adjacent to $x$. The matrix tree theorem \cite{MR1271140} states that the number of spanning trees in $G$, $\tau(G)$, is given by
\begin{equation*}
\tau(G)=\frac{\prod_{k=1}^{\lvert V(G)\rvert-1}\lambda_k}{\lvert V(G)\rvert}
\end{equation*}
where $\lambda_k$, $k=1,\ldots,\lvert V(G)\rvert-1$, are the non-zero eigenvalues of $\Delta_G$. The eigenvectors of the Laplacian on the circulant graph $C^{1,\ldots,n}_{\beta n}$ are given by the characters $\chi_k(x)=e^{2\pi ikx/(\beta n)}$, $k=0,1,\ldots,\beta n-1$. Therefore the non-zero eigenvalues are given by
\begin{equation*}
\lambda_k=2n-2\sum_{m=1}^n\cos(2\pi km/(\beta n)),\quad k=1,\ldots,\beta n-1.
\end{equation*}
Similarly, the non-zero eigenvalues on $C^{1,\ldots,n-1}_{\beta n}$ are given by
\begin{equation*}
\lambda_k=2(n-1)-2\sum_{m=1}^{n-1}\cos(2\pi km/(\beta n)),\quad k=1,\ldots,\beta n-1.
\end{equation*}
Figure \ref{powergraphs} below illustrates two power graphs of the $24$-cycle.
\begin{figure}[H]
\centering
\subfigure[$\boldsymbol{C}^8_{24}$]{\includegraphics[width=5cm]{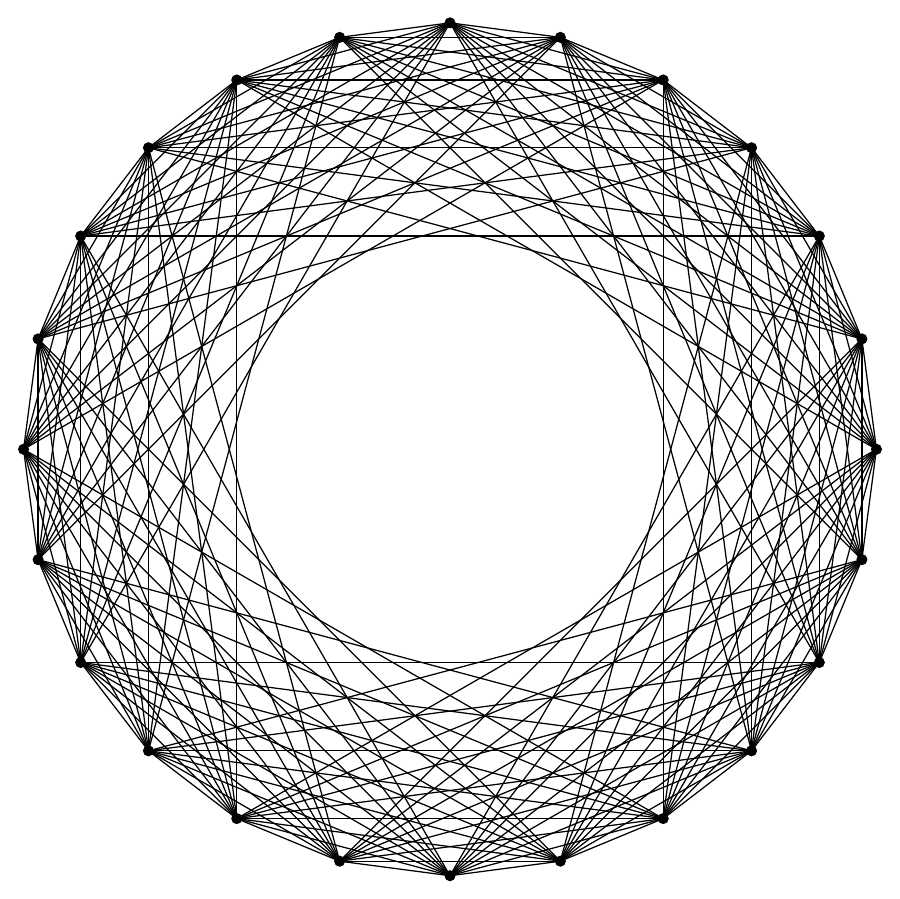}}
\hspace{2cm}
\subfigure[$\boldsymbol{C}^7_{24}$]{\includegraphics[width=5cm]{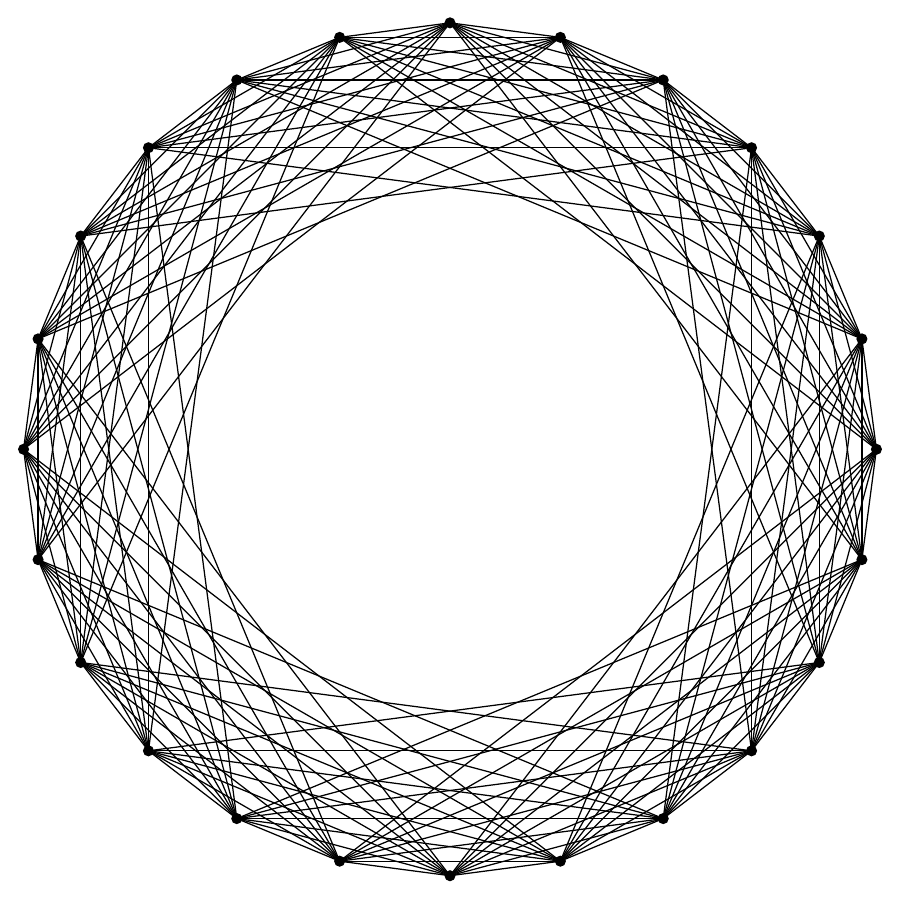}}
\caption{$8$-th and $7$-th power graphs of the $24$-cycle}
\label{powergraphs}
\end{figure}
\begin{theorem}
\label{circ}
Let $\beta\geqslant2$ be an integer and $\mu_k=2-2\cos(2\pi k/\beta)$, $k=1,\ldots,\beta-1$, be the non-zero eigenvalues of the Laplacian on the $\beta$-cycle. The number of spanning trees in the $n$-th power graph of the $\beta n$-cycle $\boldsymbol{C}^n_{\beta n}$ for $\beta\geqslant3$, is given by
\begin{align*}
\tau(\boldsymbol{C}^n_{\beta n})&=\frac{2^{\beta(n+1)}}{(2\beta)^2}n^{\beta n-2}\left(1+\frac{1}{2n}\right)^{\beta n}(1-(2n+1)^{-\beta})^n\\
&\times\prod_{k=1}^{\lceil\beta/2\rceil-1}\sin^2\left(\frac{\pi(n+1)k}{\beta}-n\Arcsin\left(\frac{n+1}{\sqrt{4n^2/\mu_k+2n+1}}\right)\right)
\end{align*}
where $\lceil x\rceil$ denotes the smallest integer greater or equal to $x$. For $\beta=2$, it is given by
\begin{equation*}
\tau(\boldsymbol{C}^n_{2n})=(2n)^{2n-2}(1+1/n)^n.
\end{equation*}
The number of spanning trees in the $(n-1)$-th power graph of the $\beta n$-cycle $\boldsymbol{C}^{n-1}_{\beta n}$, for $\beta\geqslant3$, is given by
\begin{align*}
\tau(\boldsymbol{C}^{n-1}_{\beta n})&=\frac{2^{\beta(n+1)}}{(2\beta)^2}n^{\beta n-2}\left(1-\frac{1}{2n}\right)^{\beta n}\lvert(-1)^\beta-(2n-1)^{-\beta}\rvert^n\\
&\times\prod_{k=1}^{\lceil\beta/2\rceil-1}\sin^2\left(\frac{\pi(n-1)k}{\beta}-n\Arcsin\left(\frac{n-1}{\sqrt{4n^2/\mu_k-(2n-1)}}\right)\right).
\end{align*}
For $\beta=2$, it is given by
\begin{equation*}
\tau(\boldsymbol{C}^{n-1}_{2n})=(2n)^{2n-2}(1-1/n)^n.
\end{equation*}
\end{theorem}
\begin{remark}
We emphasise that in the cycle power graphs $\boldsymbol{C}^{n-1}_{\beta n}$ and $\boldsymbol{C}^n_{\beta n}$ there are $\beta$ copies of $n$-cliques as subgraphs of the original graph. This fact appears in the formula by the factor $n^{\beta n-2}=(n^{n-2})^\beta n^{2(\beta-1)}$ since the number of spanning trees in the complete graph on $n$ vertices is $n^{n-2}$.
\end{remark}
\begin{proof}
We prove the theorem only for the first type of graphs $\boldsymbol{C}^n_{\beta n}$. The proof of the second type $\boldsymbol{C}^{n-1}_{\beta n}$ is very similar to the first one. The matrix tree theorem states that
\begin{equation*}
\tau(\boldsymbol{C}^n_{\beta n})=\frac{1}{\beta n}\prod_{k=1}^{\beta n-1}(2n-2\sum_{m=1}^n\cos(2\pi km/(\beta n))).
\end{equation*}
Lagrange's trigonometric identity expresses the sum of cosines appearing in the above formula in terms of a quotient of sines:
\begin{equation*}
2\sum_{m=1}^n\cos(2\pi km/(\beta n))=\frac{\sin((n+1/2)2\pi k/(\beta n))}{\sin(\pi k/(\beta n))}-1.
\end{equation*}
Hence,
\begin{equation*}
\tau(\boldsymbol{C}^n_{\beta n})=\frac{1}{\beta n}\prod_{k=1}^{\beta n-1}(\sin(\pi k/(\beta n)))^{-1}((2n+1)\sin(\pi k/(\beta n))-\sin(\pi k/(\beta n)+2\pi k/\beta)).
\end{equation*}
Using that there are $\beta n$ spanning trees in the $\beta n$-cycle, that is $\frac{1}{\beta n}\prod_{k=1}^{\beta n-1}(2-2\cos(2\pi k/(\beta n)))=\beta n$, it follows that
\begin{equation}
\label{taucycle}
\prod_{k=1}^{\beta n-1}\sin(\pi k/(\beta n))=\frac{\beta n}{2^{\beta n-1}}.
\end{equation}
For the second factor, as in the proof of Theorem \ref{dicd}, we split the product over $k=1,\ldots,\beta n-1$ into two products, first when $k$ is a multiple of $\beta$, that is $k=l\beta$ with $l=1,\ldots,n-1$, and second when $k$ is not a multiple of $\beta$, that is, $k=k'+l'\beta$ with $k'=1,\ldots,\beta-1$ and $l'=0,1,\ldots,n-1$. The product over the multiples of $\beta$ reduces to
\begin{equation*}
\prod_{l=1}^{n-1}2n\sin(\pi l/n)=n^n.
\end{equation*}
We have
\begin{equation}
\label{2prod}
\tau(\boldsymbol{C}^n_{\beta n})=\frac{2^{\beta n-1}n^n}{(\beta n)^2}\prod_{k=1}^{\beta-1}\prod_{l=0}^{n-1}((2n+1)\sin(\pi k/(\beta n)+\pi l/n)-\sin(\pi k/(\beta n)+\pi l/n+2\pi k/\beta).
\end{equation}
The difference of sines in the above product can be written as
\begin{equation}
\label{sine}
(2n+1)\sin(\pi k/(\beta n)+\pi l/n)-\sin(\pi k/(\beta n)+\pi l/n+2\pi k/\beta)=\lvert z_k\rvert\sin(\pi(n+1)k/(\beta n)+\theta_k+\pi l/n)
\end{equation}
where
\begin{equation*}
z_k=2n\cos(\pi k/\beta)-i(2n+2)\sin(\pi k/\beta)\eqdef\lvert z_k\rvert e^{i\theta_k}.
\end{equation*}
Let $\omega_k=\pi(n+1)k/(\beta n)+\theta_k$, we have
\begin{align}
\prod_{l=0}^{n-1}\sin(\omega_k+\pi l/n)&=\frac{1}{(2i)^n}\prod_{l=0}^{n-1}(e^{i(\omega_k+\pi l/n)}-e^{-i(\omega_k+\pi l/n)})\nonumber\\
&=\frac{1}{(2i)^n}e^{-i\omega_kn}e^{\pi i(n-1)/2}\prod_{l=0}^{n-1}(e^{2i\omega_k}-e^{-2\pi il/n})\nonumber\\
&=\frac{\sin(\omega_kn)}{2^{n-1}}
\label{prodsines}
\end{align}
where in the last equality we used equation (\ref{unityroots}). Putting equations (\ref{2prod}), (\ref{sine}) and (\ref{prodsines}) together yields
\begin{equation*}
\tau(\boldsymbol{C}^n_{\beta n})=\frac{2^{\beta n-1}n^n}{(\beta n)^2}\prod_{k=1}^{\beta-1}\frac{\lvert z_k\rvert^n}{2^{n-1}}\sin(\pi(n+1)k/\beta+n\theta_k).
\end{equation*}
Notice that for even $\beta$, the phase of $z_{\beta/2}$ is $\theta_{\beta/2}=-\pi/2$, so that $\sin(\pi(n+1)/2+n\theta_{\beta/2})=1$. For $\beta=2$, $z_1=-2(n+1)i$, hence
\begin{equation*}
\tau(\boldsymbol{C}^n_{2n})=(2n)^{2n-2}(1+1/n)^n.
\end{equation*}
For $\beta\geqslant3$, we have
\begin{equation*}
\tau(\boldsymbol{C}^n_{\beta n})=\frac{2^{n+\beta-2}n^n}{(\beta n)^2}\big(\prod_{k=1}^{\beta-1}\lvert z_k\rvert^n\big)\prod_{k=1}^{\lceil\beta/2\rceil-1}\sin(\pi(n+1)k/\beta+n\theta_k)\sin(\pi(n+1)(\beta-k)/\beta+n\theta_{\beta-k}).
\end{equation*}
For $1\leqslant k\leqslant\lceil\beta/2\rceil-1$, the phase of $z_k$ is $\theta_k=-\Arcsin((2n+2)\sin(\pi k/\beta)/\lvert z_k\rvert)$. The phase of $z_{\beta-k}$ satisfies
\begin{equation*}
\cos\theta_{\beta-k}=-\cos\theta_k,\quad\sin\theta_{\beta-k}=\sin\theta_k
\end{equation*}
so that, $\theta_{\beta-k}=\pi-\theta_k$. The modulus of $z_k$ is given by
\begin{equation*}
\lvert z_k\rvert=((2n+1)^2+1-2(2n+1)\cos(2\pi k/\beta))^{1/2}=(4n^2+(2n+1)\mu_k)^{1/2}
\end{equation*}
where $\mu_k=2-2\cos(2\pi k/\beta)$, $k=1,\ldots,\beta-1$, are the non-zero eigenvalues of the Laplacian on the $\beta$-cycle. We have $\sin(\pi k/\beta)=\mu_k^{1/2}/2$. Hence for $1\leqslant k\leqslant\lceil\beta/2\rceil-1$, the phase is given by $\theta_k=-\Arcsin((n+1)/\sqrt{4n^2/\mu_k+2n+1})$. Therefore
\begin{equation}
\label{taucirc}
\tau(\boldsymbol{C}^n_{\beta n})=\frac{2^{n+\beta-2}n^n}{(\beta n)^2}\big(\prod_{k=1}^{\beta-1}\lvert z_k\rvert^n\big)\prod_{k=1}^{\lceil\beta/2\rceil-1}\sin^2\left(\frac{\pi(n+1)k}{\beta}-n\Arcsin\left(\frac{(n+1)}{\sqrt{4n^2/\mu_k+2n+1}}\right)\right).
\end{equation}
The product of the modulus of $z_k$ is given by
\begin{align}
\prod_{k=1}^{\beta-1}\lvert z_k\rvert&=\frac{(2n+1)^{\beta/2}}{2n}\prod_{k=0}^{\beta-1}(2n+1+1/(2n+1)-2\cos(2\pi k/\beta))^{1/2}\nonumber\\
&=\frac{(2n+1)^{\beta/2}}{2n}(2\cosh(\beta\argcosh(n+1/2+1/(4n+2)))-2)^{1/2}\nonumber\\
&=\frac{(2n+1)^{\beta}}{2n}(1-(2n+1)^{-\beta})
\label{prod_modulus}
\end{align}
where the second equality comes from the identity (see \cite[section $2$]{louis2015formula})
\begin{equation*}
\prod_{k=0}^{\beta-1}(2\cosh\theta-2\cos(2\pi k/n))=2\cosh(\beta\theta)-2.
\end{equation*}
Putting equality (\ref{prod_modulus}) into (\ref{taucirc}) gives the theorem.
\end{proof}
\begin{remark}
We point out that the proof above could not be easily applied to other powers of the $\beta n$-cycle, like $\boldsymbol{C}^{n-p}_{\beta n}$, where $p\geqslant2$ or $p\leqslant-1$, because in this case $z_k$ defined in equation (\ref{sine}) would also depend on $l$ and the phase $\theta_k$ of $z_k$ cannot be easily determined. As a consequence, the product over $l$ cannot be evaluated in the same way as it is done in the proof. It would be interesting to find a derivation in this class of more general circulant graphs.
\end{remark}
From Theorem \ref{circ}, we derive the asymptotic behaviour of the number of spanning trees in the $n$-th, respectively $(n-1)$-th, power graph of the $\beta n$-cycle as $n\rightarrow\infty$.
\begin{corollary}
Let $\beta\in\mathbb{N}_{\geqslant2}$. The asymptotic number of spanning trees in the $n$-th and $(n-1)$-th power graphs of the $\beta n$-cycle $\boldsymbol{C}^n_{\beta n}$ and $\boldsymbol{C}^{n-1}_{\beta n}$ as $n\rightarrow\infty$ is respectively given by
\begin{equation*}
\tau(\boldsymbol{C}^n_{\beta n})=\frac{2^{\beta n}}{2\beta}n^{\beta n-2}(e^{\beta/2}+o(1))
\end{equation*}
and
\begin{equation*}
\tau(\boldsymbol{C}^{n-1}_{\beta n})=\frac{2^{\beta n}}{2\beta}n^{\beta n-2}(e^{-\beta/2}+o(1)).
\end{equation*}
\end{corollary}
\begin{proof}
By observing that for all $k\in\{1,\ldots,\lceil\beta/2\rceil-1\}$,
\begin{equation*}
\lim_{n\rightarrow\infty}\frac{n+1}{\sqrt{4n^2/\mu_k+2n+1}}=\sin(\pi k/\beta)\quad\textnormal{and}\quad\lim_{n\rightarrow\infty}\frac{n-1}{\sqrt{4n^2/\mu_k-(2n-1)}}=\sin(\pi k/\beta)
\end{equation*}
where $\mu_k=2-2\cos(2\pi k/\beta)$ and using relation (\ref{taucycle}) the corollary is a direct consequence of Theorem \ref{circ}.
\end{proof}

\nocite{*}
\bibliographystyle{plain}
\bibliography{bibliographyDic}

\end{document}